\documentclass[final,1p,twocolumn]{elsarticle}

\usepackage{latexsym,amssymb,amsmath}

\newcommand{\A}{\mathbb{A}}     
\newcommand{\X}{\mathbb{X}}      

\newcommand {\N} {\rm I\!N}

\newcommand{\ve}{\varepsilon}

\newcommand{\nexto}{\kern -0.54em}
\newcommand{\dR}{{\rm {I\ \nexto R}}}

\newcommand{\dZ}{{\cal Z \kern -0.7em Z}}
\newcommand{\dC}{{\rm\hbox{C \kern-0.8em\raise0.2ex\hbox{\vrule height5.4pt width0.7pt}}}}
\newcommand{\dQ}{{\rm\hbox{Q \kern-0.85em\raise0.25ex\hbox{\vrule height5.4pt width0.7pt}}}}

\newcommand{\ip}[2]{\langle #1,#2 \rangle}

\newcommand{\proofbox}{\hspace{\fill}{$\Box$}}

\newtheorem{theorem}{Theorem}
\newtheorem{corollary}{Corollary}
\newtheorem{proposition}{Proposition}

\newtheorem{definition}{Definition}

\newenvironment{proof}{{\bf Proof.}}{\proofbox}

\newcommand{\beq}{\begin{equation}}
\newcommand{\deq}{\end{equation}}
\newcommand{\baq}{\begin{eqnarray}}
\newcommand{\daq}{\end{eqnarray}}

\newcommand{\baqm}{\begin{eqnarray*}}
\newcommand{\daqm}{\end{eqnarray*}}

\newcommand{\BYDEF}{\,\shortstack{\textup{\tiny def} \\ = }\,}

\journal{Operations Research Letters}

\begin{document}

\begin{frontmatter}



\title{Singularly perturbed linear programs\\
and Markov decision processes}


\author[Inria]{Konstantin Avrachenkov}
\address[Inria]{Inria Sophia Antipolis, 2004 Route des Lucioles, 06902 Sophia Antipolis, France;
{\tt K.Avrachenkov@inria.fr}}
\author[Flinders]{Jerzy A. Filar}
\address[Flinders]{School of Computer Science, Engineering \& Mathematics, Flinders University, Australia}
\author[Macquarie]{Vladimir Gaitsgory}
\address[Macquarie]{Department of Mathematics, Macquarie University, Australia}
\author[Flinders]{Andrew Stillman}

\begin{abstract}
Linear programming formulations for the discounted and long-run average MDPs have evolved along separate trajectories. In 2006, E. Altman conjectured that the two linear programming formulations of discounted and long-run average MDPs are, most likely, a manifestation of general properties of singularly perturbed linear programs.  In this note we demonstrate that this is, indeed, the case.
\end{abstract}


\end{frontmatter}

\section{Introduction}

The connection between linear programming and Markov Decision Processes (MDPs) was launch\-ed in the 1960's, with the papers by D'Epenoux \cite{DE60}, De Ghellinck \cite{DG60} and Manne \cite{M60}. While the linear programming formulation for the discounted MDP was relatively straightforward, extension to the long-run average, multi-chain, MDP proved challenging and required nearly two decades to arrive at a single linear program supplied, by Hordijk and Kallenberg \cite{HK79,HK84}, that completely solves such a multi-chain MDP. We refer the reader to Kallenberg \cite{K83,K02}, Puterman \cite{Puterman} and Altman \cite{A99} for excellent, comprehensive, treatments of linear programming methods for discrete time Markov decision processes. Even though the approaches to discounted and long-run average MDPs evolved along separate trajectories, Tauberian theorems provided a theoretical connection between the two cases with the discount parameter approaching unity from below; e.g. see Blackwell \cite{B62} and Veinott \cite{V66,V69}.

Parametric linear programming has a long history that is well documented in many excellent textbooks (e.g., see Murty \cite{M83}). However, majority of the so-called sensitivity analyses presented in operations research books focus on perturbations of the objective function coefficients or of the right hand side vector; sometimes extending also to changes in non-basic columns.  To the best of our knowledge, Jeroslow \cite{J73} was, perhaps, the first to consider perturbations of the entire coefficient matrix of a linear program. In the context of MDP, the results
of \cite{J73} have been applied to Blackwell optimality \cite{HDK85} and to perturbed MDPs \cite{AAF99}. In Pervozvanskii and Gaitsgori \cite{PG88} the authors focus on the singularly perturbed case where a discontinuity can arise as the perturbation parameter approaches a critical value. In the latter and in the more recent book by Avrachenkov et al \cite{AFH13} the main cause of that discontinuity has been the change in the rank of the coefficient matrix at the critical value of the perturbation parameter. Hence, it was perhaps surprising that such discontinuities can also arise when the rank does not change, as shown very recently in \cite{ABFG10}.

In 2006, Eitan Altman conjectured that the two linear programming formulations of discounted and long-run average MDPs must be a manifestation of some general properties of singularly perturbed linear programs.  In this note we demonstrate that this is, indeed, the case by first extending the results in \cite{ABFG10} and then formally applying new singular perturbation results to the MDP problem.

\section{General perturbed linear programming problem}


Consider the family of linear programming problems parameterized by
$\ve>0$:
\begin{equation}\label{pro-eps}
\begin{array}{c}
\max \ip{c^{(0)}+\ve c^{(1)}}{x}\\
\hbox{ s. t. } (A^{(0)}+\ve A^{(1)}) x = b^{(0)}+\ve b^{(1)},\\
\hbox{ \hspace{4em} } x\ge 0,
\end{array}
\end{equation}
where $c^{(0)},c^{(1)}\in \dR^n$, $b^{(0)},b^{(1)}\in \dR^m$ and
$A^{(0)},A^{(1)}\in \dR^{m\times n}$. The optimal value, the
solution set and the feasible set of Problem \eqref{pro-eps} are
denoted as $F^*(\ve), \theta^*(\ve)$ and $\theta(\ve)$,
respectively.

The goal of the perturbed linear programming approach is to construct,
if possible, a linear programming problem that does not depend on $\ve$
and such that its optimal solutions are {\it feasible limiting optimal}
for (\ref{pro-eps}) in the sense prescribed below by
Definition \ref{def-FLO}. The linear program with this property
will be called a {\em limiting LP}.

\begin{definition}\label{def-FLO} A vector $x\in \dR^n$ is called
{\em feasible limiting optimal} for the perturbed linear program
$(\ref{pro-eps})$ if $\ x \in \liminf_{\ve \downarrow 0}\theta(\ve)\ $
and $\ \lim_{\ve \downarrow 0}F^*(\ve)=\ip{c^{(0)}}{x}.\ $
\end{definition}


\medskip

\noindent Let us introduce and discuss a set of assumptions:

\medskip

\noindent {\bf Assumption ($H_0$):} {\it There exists a positive $\gamma_0$
and a bounded set $B \subset \dR^n$ such that
$\theta(\ve)\subset B$ for every $\ve\in (0,\gamma_0]$.}

\medskip

\noindent {\bf Assumption ($H_0^*$):} {\it There exists a positive $\gamma_0$
and a bounded set $B \subset \dR^n$ such that
$\theta^*(\ve)\subset B$ for every $\ve\in (0,\gamma_0]$.}

\medskip

\noindent {\bf Assumption ($H_1$):} {\it The matrix $A^{(0)}$ has rank $m$.}

\medskip

\noindent {\bf Assumption ($H_2$):} {\it For all $\ve$ sufficiently small
and positive, the rank of $A^{(0)} + \ve A^{(1)}$ is equal to $m$.}

\medskip

\noindent Note that Assumption ($H_1$) implies Assumption ($H_2$).
Also, Assumption ($H_0$) implies Assumption ($H_0^*$).

\medskip

The unperturbed problem is said to satisfy Slater condition if
\begin{equation}\label{e:Slater}
\theta(0)\cap \dR^n_{++}\neq\emptyset \ , \hbox{ where }\,\,
\dR^n_{++}\BYDEF \{x\in \dR^n \ :  \ x>0\} .
\end{equation}
In \cite{PG88}, it has been shown that if Assumptions ($H_0$) and
($H_1$) are valid and if the Slater condition (\ref{e:Slater}) is
satisfied, then the unperturbed LP is the limiting problem for the
perturbed program (\ref{pro-eps}). That is, every optimal solution
of the former is limiting optimal for the latter. In \cite{PG88}
it has also been shown that if Assumption ($H_1$) is not satisfied,
the discontinuity of $\theta(\ve)$ at $\ve = 0$ may occur. This
is a case of so-called {\em singular perturbation}. The authors
of \cite{PG88} proposed a limiting LP to deal with the case of singular
perturbation. Then, in \cite{ABFG10} it has been demonstrated that
if the Slater condition is not satisfied for the unperturbed LP,
the discontinuity of $\theta(\ve)$ at $\ve = 0$ may occur with
Assumptions ($H_0$) and ($H_1$) being satisfied. The authors
of \cite{ABFG10} have constructed a limiting LP for the case
when the Slater condition is not satisfied for the unperturbed
problem. Below we show that a result similar to that obtained in
\cite{ABFG10}  can be established with the replacement of ($H_0$) by
($H_0^*$).

Assume that ($H_1$) is satisfied and define the set
\begin{equation}\label{def-J}
J_0:=\{ i\in \{1,\ldots,n\}\::\: \exists\, x\in \theta(0) \ \hbox{
such \  that } \ x_i>0\}.
\end{equation}
According to this definition, if $j\not \in J_0$, then $x_j=0$ for
every $x\in \theta(0)$. Moreover, if $J_0 \not=\emptyset$, convexity of
$\theta(0)$ implies that there exists $\hat x\in \theta(0)$ such that
$\hat x_j>0$ for every $j\in J_0$. Note that $J_0$ can be determined
by solving $n$ independent linear programming problems $\max_{x\in
\theta(0)}x_j$, with $j =1,\ldots,n$.

Consider the following linear program
\begin{equation}\label{pro-0}
\max \{ \ip{c^{(0)}}{x^0} \ : \ x^0\in \theta_1 \}\BYDEF F_1^*,
\end{equation}
where
\begin{equation}\label{pro-0-1-1}
\theta_1\BYDEF \{x^0\ : \ \exists \ (x^0,x^1)\in \Theta_1 \},
\end{equation}
and
\begin{equation}\label{pro-0-1}
 \Theta_1 = \{(x^0,x^1)\in \dR^n\times \dR^n\ : \  x^0\in \theta(0), \ \ A^{(0)}x^1+A^{(1)}x^0=b^{(1)},
 \ \ x_j^1\ge 0 \ \forall j\not\in J_0 \}.
\end{equation}
Note that,
$$
\theta_1 \subset \theta(0) \ \ \ \hbox{ and therefore }\ \ \ \ F_1^*\leq
F^*(0).
$$

Slater condition (\ref{e:Slater}) is equivalent to having
$J_0 = \{1,2,\ldots,n\}$. If this is the case, then $\ \theta_1 =
\theta(0)\ $ (provided that Assumption ($H_1 $) is satisfied), and
the problem (\ref{pro-0}) is equivalent to the unperturbed
problem. If the Slater condition is not satisfied, these two
problems are not equivalent.

Following \cite{ABFG10}, let us
introduce the following extended version of the Slater condition.

\begin{definition}\label{Def-ES-1}
We shall say that the {\em extended Slater condition of order 1}
$($or, for brevity, ES-$1$ $)$  is satisfied if there exists
$(\hat x^0,\hat x^1)\in \Theta_1$ such that $\hat x^1_j>0$ for every
$j\not\in J_0$ and $\hat x^0_j> 0$ for every $j\in J_0$.
\end{definition}


\begin{theorem}\label{T-main-1}
Let Assumptions $(H_0^*)$ and  $(H_2)$ be satisfied. Then
\begin{equation}\label{e:basic-sol-new-7}
\limsup_{\ve \downarrow 0}\theta^*(\ve)\subset \theta_1
\end{equation}
and
\begin{equation}\label{e:basic-sol-new-7-1}
 \limsup_{\ve \downarrow 0} F^*(\ve)\leq F^*_1
\ .
\end{equation}
If, in addition, Assumption $(H_1)$  and the ES-$1$
condition are satisfied, then
\begin{equation}\label{e:basic-sol-new-8}
\limsup_{\ve \downarrow 0}\theta^*(\ve)\subset \theta_1^*,
\end{equation}
where $\theta_1^* $ is the set of optimal solutions of problem (\ref{pro-0}),
and
\begin{equation}\label{e:basic-sol-new-8-1-1}
\lim_{\ve \downarrow 0} F^*(\ve)=F^*_1 \ .
\end{equation}
Also, any optimal solution $\ x^0 $ of the problem
$(\ref{pro-0})$  is limiting optimal  for the
perturbed problem $(\ref{pro-eps})$.
\end{theorem}
\begin{proof}
Most steps of the proof are similar to the corresponding steps of the proof of Theorem 2.1 in \cite{ABFG10}, and we will
only indicate the steps that  differ from those used in the aforementioned  proof.

Let us introduce the following notations.
Given a finite set $S$, denote by $|S|$ the number of elements of $S$.
Let $S_m:=\{J\subset \{1,2,\ldots,n\}\::\: |J|=m \}$,
so $|S_m|={n\choose m}$. Given a matrix $D\in \dR^{m\times
n}$ and an index set $J\in S_m$, the matrix $D_J\in \dR^{m\times
m}$ is constructed by extracting from $D$ the set of $m$ columns
indexed by the elements of $J$. In a similar way, given a vector
$x\in \dR^n$ and $J\in S_m$, we denote by $x_J$ the vector of
$\dR^m$ constructed by extracting from $x$ the coordinates $x_j,
\,j\in J$ (that is, $x_J \BYDEF \{ x_j\},\ j\in J $).

In Lemmas 3.1 and 3.2 of \cite{ABFG10} it was established that
$$
S_m=\Omega_1\cup \Omega_2 \ \ \ \ \ {\rm with} \ \ \ \ \ \Omega_1\cap \Omega_2=\emptyset ,
$$
where $\Omega_1 $ and  $\Omega_2 $ are defined by the equations
 $$\
 \Omega_1:=\{J\in S_m\::\: (A^{(0)}+\ve A^{(1)})_J \hbox{ is
nonsingular for } \ve \in (0,\gamma)\}\neq \emptyset ,
$$
$$\
\Omega_2:=\{J\in S_m\::\: (A^{(0)}+\ve A^{(1)})_J \hbox{ is
singular for all } \ve \in [0,\gamma)\} $$
(here and in what follows, $\gamma $ stands for a  positive number small enough).

Also, it was established that, if
\begin{equation}\label{e:basic-sol-new}
x_J(\ve):=[(A^{(0)}+\ve A^{(1)})_J]^{-1}(b_0+\ve b_1)
\end{equation}
($J\in \Omega_1$) and if
 \begin{equation}\label{e:bounded}
\limsup_{\ve \downarrow 0} \|x_J(\ve) \|<\infty ,
 \end{equation}
then $x_J(\ve) $ allows the power series expansion
\begin{equation}\label{t3}
x_J(\ve)=\sum_{l=0}^{\infty} \ve^l x^l_J \ , \ \ \ \ \ \ \ \forall \ve\in (0,\gamma).
\end{equation}
Let $\Omega_1^*\subset \Omega_1   $ be such that $J\in \Omega_1^* $ if and only if
there exists a subsequence $\ve'\rightarrow 0 $ such that the vector $x(\ve)=\{ x_j(\ve)\}, \ j=1,...,n $, the non-zero elements of which are equal to
the corresponding non-zero elements of  $x_J(\ve))=\{ x_j(\ve)\}, \ j\in J$ (with $x_J(\ve)$  being  as in (\ref{e:basic-sol-new})) satisfies the inclusion
$$
x(\ve')\in \theta^* (\ve').
$$
Since (due to Assumption $(H_0^*) $) (\ref{e:bounded}) is satisfied, $x_J(\ve))$ allows the expansion (\ref{t3}), and hence
\begin{equation}\label{t3-1}
x(\ve)=\sum_{l=0}^{\infty} \ve^l x^l \ , \ \ \ \ \ \ \ \forall \ve\in (0,\gamma),
\end{equation}
where non-basic components ($j\notin J $) are equal to zero in both left and right hand sides. From (\ref{t3-1}) it follows, in particular, that
\begin{equation}\label{t3-2}
\lim_{\ve \rightarrow 0}x(\ve)= x^0
\end{equation}
By substituting (\ref{t3-1}) into the constraints of the perturbed problem (\ref{pro-eps}), one can readily verify that
$(x^0, x^1)\in \Theta _1$. Hence $x^0\in \theta_1 $.

The argument above proves that any partial limit (cluster point) of any basic optimal solution of the problem  is contained
in $ \theta_1$. Since any element of $\theta^*(\ve) $ can be presented as a convex combination of the optimal basic solutions and since
$ \theta_1$ is convex, this proves the validity of (\ref{e:basic-sol-new-7}), which, in turn, implies (\ref{e:basic-sol-new-7-1}).

Let us now establish the validity of the second part of the theorem. Let $x^0\in \theta_1^* $
and let $x^1 $ be such that $(x^0, x^1)\in \Theta_1 $. Define
$(x^0(\delta),x^1(\delta))$ by the equation
\begin{equation}\label{e:basic-sol-new-9}
x^0(\delta)\BYDEF (1-\delta) x^0 + \delta \hat{x}^0, \ \ \ \ \ \
x^1(\delta)\BYDEF (1-\delta) x^1 + \delta \hat{x}^1, \ \ \ \ \ \
\delta \in (0,1),
\end{equation}
where $(\hat{x}^0, \hat{x}^1 ) $ are as in the ES-$1$ condition.
Note that $(x^0(\delta),x^1(\delta))\in \Theta_1$ (due to
convexity of $\Theta_1$) and also that
\begin{equation}\label{e:basic-sol-new-9-1}
x^0_j(\delta)\geq \delta \hat{x}^0_j \geq \delta a \ \ \forall
j\in J_0 \ , \ \ \   \ \ \ \ \ \ \ x^1_j(\delta) \geq \delta
\hat{x}^1_j\geq \delta a \ \ \forall j\notin J_0,
\end{equation}
where
\begin{equation}\label{e:basic-sol-new-9-2}
a \BYDEF \min\{\min_{j'\in J_0}\hat{x}^0_{j'}, \min_{j'\notin
J_0}\hat{x}^1_{j'}\} > 0.
\end{equation}
In the proof of Theorem 2.1 in \cite{ABFG10}, it has been established that there exists $x^2(\delta,\ve)$ such that
\begin{equation}\label{e-x-2-epsilon}
x(\delta,\ve)\BYDEF x^0(\delta) +\ve x^1(\delta) +\ve^2
x^2(\delta,\ve)\in \theta(\ve) \ \ \ \ \ \ \  \forall \delta\in [c_1\ve , 1), \ \ \ \  \forall \ve\in (0,\gamma)
\end{equation}
and such that
\begin{equation}\label{e:basic-sol-new-9-2-11}
||x^2(\delta,\ve)||\leq c_2  \ \ \ \ \forall \delta\in (0,1), \ \ \ \forall \ve\in (0,\gamma),
\end{equation}
where $c_1$ and $c_2$ are sufficiently large constants. Take $\delta(\ve)\BYDEF c_1\ve $. Then
\begin{equation}\label{e:basic-sol-new-9-2-12}
\tilde x(\ve)\BYDEF x(\delta(\ve),\ve)\in \theta(\ve)
 \ \ \ \forall \ve\in (0,\gamma),
\end{equation}
and
\begin{equation}\label{e:basic-sol-new-9-2-13}
\lim_{\ve\rightarrow 0}\tilde x(\ve)= x^0.
\end{equation}
Since
$$
 \ip{c^{(0)}+\ve c^{(1)}}{\tilde x(\ve)}\leq F^*(\ve),
$$
from (\ref{e:basic-sol-new-9-2-13}) it follows that
\begin{equation}\label{e:basic-sol-new-9-2-14}
F_1^* = \ip{c^{(0)}}{x^0}\leq \liminf_{\ve\rightarrow 0}F^*(\ve)
\end{equation}
(the equality being due to the fact that $ x^0 $ was chosen to be an optimal solution of (\ref{pro-0})).
The validity of (\ref{e:basic-sol-new-9-2-14}) and (\ref{e:basic-sol-new-7-1}) implies the validity of (\ref{e:basic-sol-new-8-1-1}). The latter along with
(\ref{e:basic-sol-new-7}) imply (\ref{e:basic-sol-new-8}). Finally, the fact that any optimal solution $x^0$ of the problem (\ref{pro-0})
is limiting optimal in the perturbed problem (\ref{pro-eps}) follows from (\ref{e:basic-sol-new-9-2-13}) and from that
$$
\lim_{\ve\rightarrow 0}\ip{c^{(0)}+\ve c^{(1)}}{\tilde x(\ve)}=\ip{c^{(0)}}{x^0}= F^*_1= \lim_{\ve\rightarrow 0}F^*(\ve).
$$
\end{proof}

\medskip

Instead of problem~\eqref{pro-0}, it may be more convenient to deal with the
following problem
\begin{equation}\label{pro-1}
\begin{array}{l}
\max \ip{c^{(0)}}{x^0}\\
\hbox{ s. t. } A^{(0)}x^0=b^{(0)},\\
\hbox{       } A^{(0)}x^1+A^{(1)}x^0=b^{(1)},\\
\hbox{ \hspace{4em} } x^0\ge 0,\\
\hbox{ \hspace{4em} } x^1\ge 0,\\
\end{array}
\end{equation}
the statement of which does not involve the set $J_0$. Let us give a sufficient
condition, under which the problems ~\eqref{pro-0} and \eqref{pro-1} are
equivalent  in the sense of Definition \ref{def:pro} introduced below (the latter makes
use of the fact that the objective function in \eqref{pro-0} and \eqref{pro-1} do not explicitly
depend on  $x^1$).

\begin{definition}\label{def:pro}
We will say that the problems \eqref{pro-0} and \eqref{pro-1} are {\em equivalent} when
the sets $\theta_1$ and $\tilde{\theta}_1$,
\begin{equation}\label{def:tildetheta1}
\tilde\theta_1 \BYDEF \{x^0\in \theta(0)\::\: \hbox{there exists}\
x^1\in \dR^n \hbox{ such that } A^{(0)}x^1+A^{(1)}x^0=b^{(1)},\,\,x^1\ge 0\},
\end{equation}
coincide.
\end{definition}

A sufficient condition for problems \eqref{pro-0} and \eqref{pro-1} to
be equivalent is provided by the following result.

\medskip

\begin{proposition}\label{Q}
Let  $J_0$ be as in $(\ref{def-J})$. If there exists $\alpha:=\{\alpha_j\}_{j\in J_0}$ such that
\begin{equation}\label{C2}
A^{(0)}_{J_0} \alpha =0,\, \hbox{\rm with }\alpha_j>0 \,\ \ \forall\, j\in J_0.
\end{equation}
Then problems \eqref{pro-0} and \eqref{pro-1} are equivalent.
\end{proposition}
\begin{proof}
We must show that $\tilde\theta_1=\theta_1$. The definitions readily
imply that $\tilde\theta_1 \subset \theta_1$. Let us prove  the opposite inclusion.
Take $x^0\in \theta_1$. From the definition of $\theta_1$ it follows that $x^0\in \theta(0)$.
From this definition it also follows that there exists  $x^1=(x^1_j)\in \dR^n$ such that
\begin{equation}\label{C1}
x^1_j\ge 0,\,\forall\,j\not\in J_0,\, \ \  A^{(0)}x^1 + A^{(1)}x^0 =b^{(1)}.
\end{equation}
If $x^1\ge 0$, then by definition $x^0\in \tilde\theta_1$. Otherwise, there exist a component (or components) of $x^1$
 such that $x^1_j<0$ for $j\in J_0$. In this case, we can take
\[
t>\max_{j\in J_0,\ x^1_j<0}\{-x^1_j/\alpha_j\}>0,
\]
where $\{\alpha_j\}_{j\in J_0}$ are as in \eqref{C2}. Define $\hat{x}\in \dR^n$ as
\[
\hat{x}_j:=\left\{
\begin{array}{lr}
t \alpha_j + x^1_j& \hbox{ if } j\in J_0,\\
x^1_j& \hbox{ if } j\not\in J_0,
\end{array}\right.
\]
The definition of $t$ ensures $\hat{x}\ge 0$. Using  \eqref{C2} and \eqref{C1},
 we also have
\[
\begin{array}{rcl}
A^{(0)}\hat{x} + A^{(1)}x^0&=& A^{(0)}_{J_0} (t\alpha + x^1_{J_0}) +  A^{(0)}_{J^C_0} [x^1]_{J^C_0}  + A^{(1)}x^0\\
& =& t A^{(0)}_{J_0} \alpha  + A^{(0)}x^1 + A^{(1)}x^0= b^{(1)},
\end{array}
\]
where we used the notation $J^C_0:=\{i\::\: i\not\in J_0 \}$. The above
expression implies that $x^0\in \tilde\theta_1$, because we found a vector $\hat{x}\in \dR^n$ such that
$\hat{x}\ge 0$ and $A^{(0)}\hat{x} + A^{(1)}x^0=b^{(1)}$.
\end{proof}

\medskip

\begin{corollary}\label{Cor-100}
If $b^{(0)}=0 $, then problems \eqref{pro-0} and \eqref{pro-1} are equivalent.
\end{corollary}
\begin{proof}
By the very definition of $J_0$, there exists $x_{J_0}>0$ such that
$$
A^{(0)}_{J_0}x_{J_0} = b^{(0)}=0.
$$
Thus, the role of $\alpha$ is played by $x_{J_0} $ in the present case.
\end{proof}
%

\section{Application to Markov Decision Processes}

Let us consider a discrete-time Markov Decision process (also called
a Controlled Markov Chain) with a finite state space
$\X=\{1,...,N\}$ and a finite action space $\A(i)=\{1,...,m_i\}$ for
each state $i \in \X$. At any time point $t$ the system is in one of
the states $i \in \X$ and the controller or ``decision-maker''
chooses an action $a \in \A(i)$; as a result the following occur:
(a) the controller gains an immediate reward $r_{ia}$, and (b) the
process moves to a state $j \in \X$ with transition probability
$p_{iaj}$, where $p_{iaj} \ge 0$ and $\sum_{j \in \X}p_{iaj}=1$.

A decision rule $\pi_t$ at time $t$ is a function which assigns a
probability to the event that any particular action $a$ is taken at
time $t$. In general, $\pi_t$ may depend on all history
$h_t=(i_0,a_0,i_1,a_1,...,a_{t-1},i_t)$ up to time $t$. The
distribution $\pi_t(a_t|h_t)$ defines the probability of selecting
the action $a_t$ at time $t$ given the history $h_t$.

A control (or policy) is a sequence of decision rules
$\pi=(\pi_0,\pi_1,...,\pi_t,...)$. A policy $\pi$ is called Markov
if $\pi_t(\cdot |h_t)=\pi_t(\cdot |i_t)$. If $\pi_t(\cdot
|i)=\pi_{t'}(\cdot |i)$ for all $t,t' \in \N$ then the Markov policy
$\pi$ is called stationary. It is defined by a distribution
$\pi_{ia}$, where $\pi_{ia}$ is the probability of choosing action
$a$ when the system is in state $i$. Furthermore, a deterministic
policy $\pi$ is a stationary policy whose single decision rule is
nonrandomized. It can be defined by the function $f(i)=a, a \in
\A(i)$.

Let $U$, $U^S$ and $U^D$ denote the sets of all policies, all
stationary policies and all deterministic policies, respectively. It
is known that, in many contexts, there is no loss of generality in
restricting consideration to stationary or even deterministic
policies (see e.g., \cite{Puterman}).

For any stationary policy $\pi \in U^S$ we can define the
corresponding transition matrix $P(\pi)=\{p_{ij}(\pi)\}_{i,j=1}^N$
and the reward vector $r(\pi)=\{r_{i}(\pi)\}_{i=1}^N$
$$
p_{ij}(\pi):=\sum_{a \in \A(i)}p_{iaj}\pi_{ia}, \quad
r_i(\pi):=\sum_{a \in \A(i)}r_{ia}\pi_{ia}.
$$
The expected average reward $g_i(\pi)$ and the expected discounted
reward $v_i^\alpha(\pi)$, associated with policy $\pi$, can be expressed as follows:
$$
g_i(\pi):=\lim_{T \to \infty} {1 \over T} \sum_{t=1}^T \left[
P^{t-1}(\pi)r(\pi) \right]_i
$$
and
$$
v_i^\alpha(\pi):=(1-\alpha)\sum_{t=1}^{\infty}\alpha^{t-1} \left[
P^{t-1}(\pi)r(\pi) \right]_i =(1-\alpha)\left[ (I-\alpha
P(\pi))^{-1} r(\pi) \right]_i ,
$$
respectively, where $i \in \X$ is an initial state and
$\alpha \in (0,1)$ is a discount factor.

Often an interest rate $\rho=(1-\alpha)/\alpha$ is used instead of
the discount factor. We note that the interest rate is close to zero
when the discount factor is close to 1.

The following power series expansion, so-called Blackwell series expansion
\cite{B62,Puterman}, helps to establish a relation
between discount optimality and average optimality

\beq \label{power_ser}
v_i^\alpha(\pi)=(1-\alpha)\left[\frac{g_i(\pi)}{1-\alpha}+h_i(\pi)+...\right]=
g_i(\pi)+(1-\alpha)h_i(\pi)+... \ , \deq

\noindent where $h(\pi)=(I-P(\pi)+P^*(\pi))^{-1}(I-P^*(\pi))r(\pi)$
is a so-called bias vector. We note that often the expected discount
reward vector is introduced without the factor $(1-\alpha)$. In that
case the power series (\ref{power_ser}) becomes a Laurent power
series. However, the factor $(1-\alpha)$ makes exposition of the results
easier in the context of our singular perturbation approach.

We now introduce the discount optimality and the average optimality
criteria in MDP optimization problem.

\begin{definition}
The stationary policy $\pi_*$ is called the discount optimal for
fixed $\alpha \in (0,1)$ if
$$ v_i^\alpha(\pi_*) \ge v_i^\alpha(\pi)$$
for each $i \in \X$ and all $\pi \in U^S$.
\end{definition}

\begin{definition}
The stationary policy $\pi_*$ is called the average optimal if
$$g_i(\pi_*) \ge g_i(\pi) $$
for each $i \in \X$ and all $\pi \in U^S$.
\end{definition}

The power series (\ref{power_ser}) suggests another equivalent
definition of the average optimality.

\begin{definition}
The stationary policy $\pi_*$ is called the average optimal if
$$\lim_{\alpha\uparrow 1} [v_i^\alpha(\pi_*)-v_i^\alpha(\pi)] \ge 0$$
for each $i \in \X$ and all $\pi \in U^S$.
\end{definition}

We note that this definition corresponds to the concept of limiting
optimality in the context of perturbed linear programming.

In the case of discount optimality, the optimal value vector can
be found as a solution of the following LP (see e.g., \cite{A99,Puterman}).
\beq \label{discLP}
\min_{\gamma} \sum_j \gamma_j \tilde v_j
\deq
$$\mbox{subject to}
\sum_j [\delta_{ij}-\alpha p_{iaj}] \tilde v_j \ge r_{ia}, \forall (i,a) \in \X \times \A,
$$
where $\gamma_j > 0$ and can be chosen as elements of some probability
distribution. Without loss of generality, we may assume that the additional
non-negativity constraints
$$
\tilde v_j\geq 0,\ \quad \forall \ j,
$$
are satisfied. The latter can be induced by adding a sufficiently large value $r_0 > 0$
to all immediate rewards $r_{ia}$. This transformation does not change the structure
of optimal policies.

In the case of long-run average optimality, the optimal value vector can be
found as a solution of another LP (see e.g., \cite{A99,Puterman}).
\beq \label{averLP}
\min_{\gamma} \sum_j \gamma_j \tilde v_j
\deq
$$\mbox{subject to}
\sum_j [\delta_{ij}- p_{iaj}] \tilde v_j \ge 0\ \ \forall (i,a) \in \X \times \A,
$$
$$
\tilde v_i + \sum_j [\delta_{ij}- p_{iaj}] \tilde u_j \ge r_{ia} \ \ \forall (i,a) \in \X \times \A.
$$
Again, by adding a sufficiently large value $r_0 > 0$ to all immediate rewards and noticing
that
\beq \label{discLP-2}
\sum_j \delta_{ij} = \sum_j p_{iaj} =1 \ \ \ \forall \ (i,a),
\deq
one may assume, without loss of generality, that the non-negativity constraints
$$
\tilde v_j\geq 0,\ \  \tilde u_j\geq 0 \ \ \forall \ j,
$$
are satisfied.

Our aim is to demonstrate that LP (\ref{averLP}) for the log-run average MDP can be derived
from LP (\ref{discLP}) for the discounted MDP by the formal singular perturbation methods \cite{ABFG10,PG88}.

Take $\ve \BYDEF(1-\alpha)/\alpha$. By making a change of variables $v_j = \ve/(1+\ve) \tilde v_j$, one can rewrite
the LP problem (\ref{discLP}) in the form
\beq \label{discLP1}
\min_{\gamma} \sum_j \gamma_j v_j
\deq
$$\mbox{subject to} \ \ \
\sum_j [(1+\ve)\delta_{ij} - p_{iaj}] v_j \ge \ve r_{ia}, \forall (i,a) \in \X \times \A,
$$
$$
v_j\geq 0, \ \ \ j=1,...,n.
$$
Since Theorem \ref{T-main-1} is stated for the linear programs with equality constraints, let us introduce
additional variables $\sigma_{ia}$ to transform (\ref{discLP1}) to
\beq \label{discLP1s}
\min_{\gamma} \sum_j \gamma_j v_j
\deq
$$\mbox{subject to}
\sum_j [(1+\ve)\delta_{ij} - p_{iaj}] v_j - \sigma_{ia} = \ve r_{ia}\ \ \forall (i,a) \in \X \times \A,
$$
$$
v_j \ge 0, \quad \sigma_{ia} \ge 0.
$$
Note that the  linear program above is just a particular case
of (\ref{pro-eps}) with $A^{(0)}=\{\delta_{ij}-P_{iaj} \ | \ -I\}$,
$A^{(1)}=\{\delta_{ij} \ | \ 0\}$, $b^{(0)}=\{0\}$, $b^{(1)}=\{r_{ia}\}$,
and $c^{(0)}=\{\gamma_j\}$, $c^{(1)}=0$. The problem (\ref{pro-1}) (which is equivalent to (\ref{pro-0})  due to the fact
that  $b^{(0)}=\{0\}$; see Corollary \ref{Cor-100}) can in this case
be written as follows
\beq \label{discLP1s-1}
\min_{\gamma} \sum_j \gamma_j v_j^0
\deq
\beq \label{discLP1s-1-11}\mbox{subject to} \ \ \
\sum_j [\delta_{ij} - p_{iaj}] v_j^0 - \sigma_{ia}^0 = 0 \ \ \forall (i,a) \in \X \times \A,
\deq
\beq \label{discLP1s-1-12}
v_i^0 + \sum_j [\delta_{ij} - p_{iaj}] v_j^1 - \sigma_{ia}^1 = r_{ia}\ \  \forall (i,a) \in \X \times \A,
\deq
$$
v_j^0 \ge 0, \quad \sigma_{ia}^0 \ge 0, \ \ \ \ v_j^1 \ge 0, \quad \sigma_{ia}^1 \ge 0.
$$
Note that this problem is equivalent to (\ref{averLP}) (with $v_j^0$ and $v_j^1$ playing the roles of $\tilde v_j$ and $\tilde u_j$ respectively).

\begin{theorem}
The  Assumptions $(H_0^*)$, $(H_1)$ and the ES-$1$ condition are satisfied and, hence,
the problem $(\ref{discLP1s-1})$ is limiting LP for the problem $(\ref{discLP1s})$
in the sense that $(\ref{e:basic-sol-new-8})$ and $(\ref{e:basic-sol-new-8-1-1})$ are satisfied.
\end{theorem}
\begin{proof} Since we consider the discounted reward vector normalized by
$1-\alpha = \ve (1+\ve)$ (see (\ref{power_ser})), the optimal value of the problem (\ref{discLP1s})
remains bounded as $\ve \to 0$.
Hence, since also $\gamma_j$ are assumed to be positive, Assumption $(H_0^*)$ is satisfied.
Assumption $(H_1)$ is obviously satisfied as well (as the matrices of constraints
contain the identity matrix). Let us now prove that the ES-$1$ condition is satisfied. Denote by $J_{0,v}$ and $J_{0,\sigma}$ the sets of indices such that from the fact  $v^0= (v^0_j)\geq 0$ and  $\sigma^0= (\sigma^0_j)\geq 0$   satisfy (\ref{discLP1s-1-11}) it follows
that
$$
v^0_j = 0 \ \ \forall j\notin J_{0,v},  \ \ \ \ \ \ \ \ \ \ \sigma^0_j = 0 \ \ \forall j\notin J_{0,\sigma}.
$$
To verify the ES-$1$  condition, one needs to show that there exist
 \begin{equation}\label{e-extra-11}
 \hat v^0= (\hat v^0_j)\geq 0, \ \ \ \ \ \hat \sigma^0= (\hat \sigma^0_j)\geq 0,\ \ \ \ \  \hat v^1= (\hat v^1_j)\geq 0, \ \ \ \ \ \hat \sigma^1= (\hat \sigma^1_j)\geq 0
 \end{equation}
 that satisfy (\ref{discLP1s-1-11}), (\ref{discLP1s-1-12}) as well as the property that
 \begin{equation}\label{e-extra-12}
 \hat v^0_j > 0 \ \ \forall j\in J_{0,v}, \ \ \ \ \ \ \  \hat \sigma^0_j > 0 \ \  \forall j\in J_{0,\sigma},\ \ \ \ \ \ \
 \hat v^1_j > 0 \ \ \forall j\notin J_{0,v}, \ \ \ \ \ \ \ \hat \sigma^1_j > 0 \ \ \forall j\notin J_{0,\sigma}.
 \end{equation}
Note that $J_{0,v}^c = \emptyset,$ due to the fact that
 $\sum_j [\delta_{ij} - p_{iaj}] M = 0$ for any $M>0$ and any pair $(i,a)$.
Also, if $v^0= (v^0_j)\geq 0, \ \sigma^0= (\sigma^0_j)\geq 0,\ v^1= (v^1_j)\geq 0, \  \sigma^1= (\sigma^1_j)\geq 0$
satisfy (\ref{discLP1s-1-11}), (\ref{discLP1s-1-12}) and the inequalities
$v^0_j > 0 \  \forall j$, $\ \sigma^0_j > 0 \   \forall j\in J_{0,\sigma},$
are valid, then
$\hat v^0, \ \hat \sigma^0, \hat v^1 , \  \hat \sigma^1$, with the components defined as follows
\begin{equation}\label{e-extra-12-1}
 \hat v^0_j\BYDEF v^0_j +M \ \  \forall j, \ \ \ \ \ \ \hat \sigma^0_j\BYDEF \sigma^0_j\ \  \forall j,\ \ \ \ \ \
  \hat v^1_j\BYDEF v^1_j \ \ \forall j, \ \ \ \ \ \ \hat \sigma^1_j \BYDEF \sigma ^1_j + M \ \  \forall j,
  \end{equation}
will satisfy (\ref{discLP1s-1-11}), (\ref{discLP1s-1-12}) and (\ref{e-extra-12}), provided that
$M$ is chosen large enough.
 This completes the proof.

\end{proof}

\section{Acknowledgements}

The work of K. Avrachenkov was partially supported by ARC Discovery Grant DP120100532
and EU Project Congas FP7-ICT-2011-8-317672; the work of J. Filar was partially
supported by the ARC grant DP150100618; and the work of V. Gaitsgory was partially
supported by the ARC Discovery Grants DP130104432, DP120100532 and DP150100618.

\end{document}